\documentclass{amsart}

\usepackage{amsmath}
\usepackage{amsthm}
\usepackage{xspace}

\newcommand{\nx}{\ensuremath{\mathrm{nx}}\xspace}
\newcommand{\Mod}[1]{\ (\mathrm{mod}\ #1)}

\theoremstyle{plain}
\newtheorem*{theorem}{Theorem}

\newtheorem*{corollary}{Corollary}

\theoremstyle{definition}
\newtheorem*{definition}{Definition}

\title[A simple and elementary proof of Whitney's unique embedding theorem]{A simple and elementary proof of Whitney's unique embedding theorem}

\author[G. Brinkmann]{Gunnar Brinkmann}
\email[G. Brinkmann]{Gunnar.Brinkmann@UGent.be}

\address [G. Brinkmann]{Applied Mathematics, Computer Science and Statistics\\
Ghent University}

\begin{document}
\begin{abstract}
  In this note we give a short and elementary proof of a more general
  version of Whitney's theorem that 3-connected planar graphs have a
  unique embedding in the plane. A consequence of the theorem is also
  that cubic plane graphs cannot be embedded in a higher genus with a
  simple dual. The aim of this paper is to promote a simple and
  elementary proof, which is especially well suited
  for lectures presenting Whitney's theorem.
  
\end{abstract}
\maketitle

\section{Introduction}

We will describe the proof in the language of combinatorial embeddings in orientable surfaces. For the translation to
the language of topological 2-cell embeddings, methods from standard books like \cite{top_graph_theory} or \cite{graphs_on_surfaces}
can be used.

We interpret each edge $\{u,v\}$ of an undirected embedded graph $G$ as two directed edges:
$e=(u,v)$ and its inverse $e^{-1}=(v,u)$.
An embedded graph is a graph where for every vertex $u$ there is a cyclic order of all edges $(u,.)$, which we interpret as clockwise.
We write $\nx(e)$ for the next edge in the order around the starting point of a directed edge $e$. The inverse graph or
mirror image is the graph $G^{-1}$ with all cyclic orders reversed.

A face in an embedded graph $G$ is a directed cyclic walk $e_0, \dots ,e_{n-1}$, so that for $0\le i <n$ we have that
$\nx(e_i^{-1})=e_{(i+1) \Mod{n}}$.
We say that the set $\{e, \nx(e)\}$ forms an {\em angle} of $G$ and $G^{-1}$ if one of them has a face
containing $e^{-1}, \nx(e)$ as a subsequence. In this case the other has a face containing $\nx(e)^{-1},e$.
If a face is a simple cyclic walk, we call the corresponding undirected cycle also a (simple) facial cycle.
We consider an embedded graph
$G$ and its mirror image $G^{-1}$ as equivalent, as the faces
have the same sets of underlying undirected edges.
The genus of an embedded graph
can be computed by the Euler formula using the number $v$ of vertices, $e$ of (undirected) edges, and $f$ of faces
as $\gamma(G)=\frac{2-(v-e+f)}{2}$.
We refer to a (not necessarily embedded) graph that can be embedded with genus 0 as planar and to a
an embedded graph with genus 0 as plane.

With this notation and concept of equivalence Whitney's famous theorem \cite{Whitney_unique_embed} can be shortly stated as:\\
{\em A 3-connected planar graph has an -- up to equivalence --  unique embedding in the plane.}

We will prove a stronger theorem using the concept of {\em polyhedral embedding} that requires
some important properties of polyhedra -- that is plane 3-connected graphs -- but allows higher
genera. It is an easy consequence of the Jordan Curve Theorem that polyhedra are polyhedral
embeddings.

\begin{definition}
A polyhedral embedding of a graph $G=(V,E)$ in an orientable surface is an embedding so that
each facial walk is a simple cycle and the intersection of any two faces is either empty, a single vertex
or a single edge.

For cubic embedded graphs this is equivalent to the dual graph being simple.
\end{definition}

The argument of crossing Jordan curves that we will use in the proof
was first published by Thomassen in \cite{Thomassen_90}, but also
known to Robertson and later used by Mohar and Robertson in
\cite{MR96}. See also Theorem~5.7.1 in \cite{graphs_on_surfaces}. In fact, in \cite{Thomassen_90} the
argument was used to prove that 3-connected planar graphs embedded
with genus $g>0$ have facewidth at most 2. Together with Whitney's theorem, this implies Theorem~\ref{thm}.
We will give every detail of the proof in order to make it easily accessible also for students, but the arguments
are exactly the same arguments of crossing Jordan curves that Thomassen used -- only that here the planar case,
that is: Whitney's theorem -- is included too.

\begin{theorem}\label{thm}

  A 3-connected planar graph has an -- up to equivalence -- unique polyhedral embedding.

 \end{theorem}

\begin{proof}

  Let $G$ be a plane embedding of a 3-connected planar graph with mirror image $G^{-1}$ and let $G'$ be
  an embedding different from these two.
  We say that a vertex of $G'$ has type $1$ if the order is the same
  as in $G$, type $-1$ if it is the same as in $G^{-1}$ and type $2$ otherwise.
  As $G'$ is neither $G$ nor $G^{-1}$, $G'$ has a vertex of type $2$ or an edge
  with one vertex of type $1$ and one vertex of type $-1$.

  Assume first that there is a vertex $v$ of type $2$. Let $e_0,\dots ,e_{d-1}$ be the order of edges around $v$ in $G'$.
  If $\{e_0,e_1\}$ is no angle of $G$, we take this set of edges. Otherwise
  assume w.l.o.g.\ that $e_1=\nx(e_0)$ in $G$ and let $j$ be minimal so that
  in $G$ we have $\nx(e_j)\not=e_{(j+1)\Mod{d}}$.
  As in $G^{-1}$ we have $\nx(e_j)=e_{j-1}$, the edge
$e_{(j+1)\Mod{d}}$ follows $e_j$ neither in $G$ nor in $G'$, so
  $\{e_j,e_{(j+1)\Mod{d}}\}$ is no angle in $G$.
  W.l.o.g.\ assume $j=0$.

  So the order around $v$ in $G$ is $e_0,e_{i_1},\dots ,e_{i_j},e_1,e_{i_{j+1}},\dots ,e_{i_{d-2}}$
  with $1\le j<d-2$ and assume w.l.o.g.\ that $e_{d-1}\in \{e_{i_{j+1}},\dots ,e_{i_{d-2}}\}$.
  Let $y= \max \{i_1,\dots ,i_j\}$, so
  $y<d-1$ and $(y+1)\in \{i_{j+1},\dots ,i_{d-2}\}$, which implies that $\{e_y,e_{y+1}\}$ is an angle of $G'$ with
$e_y\in \{e_{i_1},\dots ,e_{i_j}\}$ and $e_{y+1}\in \{e_{i_{j+1}},\dots ,e_{i_{d-2}}\}$.
    Let $F$ be the facial cycle in $G'$ containing the
  angle $\{e_0,e_1\}$ and $F'$ be the facial cycle containing $\{e_y,e_{y+1}\}$. We have $F\not=F'$ as otherwise
  the faces would not be simple cycles. In $G$ these cycles are no facial cycles, but two Jordan curves
  crossing each other in $v$. Due to the Jordan curve theorem, there must be a second crossing,
  so $F,F'$ are two facial cycles that have at least two vertices in common that are
  no endpoints of a common edge -- a contradiction to $G'$ being polyhedral.

  Assume now that all vertices are of type $1$ or type $-1$. Then there is an edge $e_0$ with one vertex
  of type $1$ and one of type $-1$. Assume that in $G$ the orientation around the type $1$ vertex of $e_0$ is
  $e_0,e_1,\dots ,e_d$ and around the type $-1$ vertex it is $e_0^{-1},e'_1,\dots ,e'_{d'}$, so in $G'$ it is
  $e_0,e_1,\dots ,e_d$ resp. $e'_{d'}, e'_{d'-1},\dots ,e_0^{-1}$. In $G'$ there is a face $F$ containing
  $e_d^{-1},e_0,e'_{d'}$ and another face $F'$ containing ${e'_1}^{-1},e_0^{-1},e_1$. In $G$ the corresponding
  cycles are again no
  facial cycles but Jordan curves crossing each other (with one common edge), so like in the first case we
  get a contradiction from the fact that there must be a second intersection between $F$ and $F'$.

  \end{proof}

As plane embeddings of 3-connected graphs are all polyhedral, this also implies Whitney's theorem, but there are
also other consequences that are worth mentioning. They follow already from Theorem~8.1 in \cite{Thomassen_90}.
Note that for graphs with 1- or 2-cut there are no polyhedral embeddings in any surface.

\begin{corollary}
\begin{itemize}
\item  There are no polyhedral embeddings of planar graphs in any orientable surface but the plane.
\item  There are no embeddings of cubic planar graphs with a simple dual in any orientable surface but the plane.
\end{itemize}
\end{corollary}

{\bf Acknowledgements:} I would like to thank Bojan Mohar for pointing me to the earlier uses of the crossing Jordan curves argument!

\bibliographystyle{plain}
%\bibliography{../literatur.bib}

\end{document}